\documentclass[11pt,twoside]{amsart}
\usepackage{amsfonts}
\usepackage{amsmath}
\usepackage{amssymb}
\usepackage{mathrsfs}

\newtheorem{theo}{Theorem}
\newtheorem{cor}{Corollary}
\newtheorem{lem}{Lemma}
\newtheorem{prop}{Proposition}
\theoremstyle{definition}
\newtheorem{defn}{Definition}

\theoremstyle{remark}
\newtheorem{rem}{\bf Remark\/}

\numberwithin{equation}{section}

\def\1{{\mathchoice {\rm 1\mskip-4mu l} {\rm 1\mskip-4mu l}{\rm 1\mskip-4.5mu l} {\rm 1\mskip-5mu l}}}
\newcommand{\ds}{\displaystyle}

\title[Energy classes  associated to a  positive closed current]{Pluricomplex energy classes associated to a  positive closed current}

\author[J. Hbil]{Jawhar Hbil}
\email{jawhar\_x@hotmail.fr}
\author[M. Zaway]{Mohamed Zaway}
\email{mohamed\_zaway@yahoo.fr}
\author[N. Ghiloufi]{Noureddine Ghiloufi}
\email{noureddine.ghiloufi@fsg.rnu.tn}
\address{Department of Mathematics\\ Faculty of sciences of Gabes \\ University of Gabes \\ Zrig 6072 Gabes Tunisia.}
\begin{document}

\begin{abstract}

The aim of this paper is to extend the domain of definition of $(dd^c\centerdot)^q\wedge T$ on some classes of plurisubharmonic (psh) functions, which are not necessary bounded, where $T$ is a positive closed current of bidimension $(q,q)$ on an open set $\Omega$ of $\Bbb C^n$. We introduce two classes  $\mathcal{F}_{p}^{T}(\Omega)$ and $\mathcal{E}_p^T(\Omega)$ and we show that they belong to the domain of definition of the operator $(dd^c\centerdot)^q\wedge T$. We also prove that all functions belong to these classes  are $C_T$-quasicontinuous and that the comparaison principle is valid in them.
\end{abstract}
\maketitle

\section{Introduction}

Let $\Omega$ be a bounded open set of $\Bbb C^n$ and denote by $PSH(\Omega)$ the set of psh functions on $\Omega$. The definition of the complex Monge-Amp\`ere operator $(dd^c\centerdot)^n$ on the set of psh functions has been studied  by Bedford and Taylor in \cite{Be-Ta}, they showed that this operator is well defined on the set of bounded psh functions and they established the comparaison principle to study the Dirichlet problem on $PSH(\Omega)\cap L^{\infty}(\Omega)$. The problem of extending its domain of definition was treated by many other authors, in particular Cegrell has introduced, between 1998 and 2004 (see \cite{Ce1,Ce2}), a general class $\mathcal{E}(\Omega)$: the class  of psh functions which are locally equal to decreasing limits of bounded psh functions vanishing on $\partial \Omega$ with bounded Monge-Amp\`ere mass on $\Omega$. He showed that the Monge-Amp\`ere operator is well defined on $\mathcal{E}(\Omega)$ and   this is the largest domain of definition if  the operator is required to be continuous under decreasing sequences. The study of this class leads to many results such that the comparaison principle, the solvability of the Dirichlet problem and the convergence in capacity.\\

Throughout this paper,   $T$ will be  a positive closed current of bidimension $(q,q)$ on $\Omega$ where $1\leq q\leq n$. The question is to extend the domain of definition of  the operator $(dd^c\centerdot)^q\wedge T$. This problem was studied by Dabbek and Elkhadhra \cite{Da-El} in the case of bounded psh functions. We will extend the domain of definition of this operator to some classes of unbounded psh functions.\\

In this paper we recall the classes $\mathcal F^T(\Omega)$ and $\mathcal E^T(\Omega)$ introduced in \cite{Ha-Du} where the Monge-Amp\`ere operator $(dd^c\centerdot)^q\wedge T$ is well defined and we introduce two new classes, the first will be  $\mathcal F_p^T(\Omega),\ p\geq1$ a subclass of  $\mathcal F^T(\Omega)$ and  the second will be $\mathcal E_p^T(\Omega)$.\\
In the first part we introduce the class $\mathcal E_p^T(\Omega)$ and we show that the Monge-Amp\`ere operator $(dd^c\centerdot)^q\wedge T$ is well defined on this class then we give some properties of the classes $\mathcal E_p^T(\Omega)$ and $\mathcal F^T(\Omega)$.\\
In the second part we prove that every functions in $\mathcal E_p^T(\Omega)$ or in  $\mathcal F^T(\Omega)$ are $C_T$-quasicontinuous; it means that they are continuous outside subsets of small $C_T$-capacity. The main tool of this result will be an estimate of the growth of $C_T(\{u<-s\})$. Indeed we prove that
$$C_T(\{u<-s\})=O\left(\frac1{s^{p+q}}\right)\quad (\hbox{resp. }C_T(\{u<-s\})=O\left(\frac1{s^q}\right))$$
for every $u\in\mathcal E_p^T(\Omega)$ (resp. $u\in\mathcal F^T(\Omega)$).

Using some analogous   Xing's inequalities, we prove in the last part the main result of this paper.

\noindent\textbf{Main result (Comparison principle)}
    \emph{Let $u\in \mathcal{F}^{T}(\Omega)$ and $v\in \mathcal E^{T}(\Omega)$. Then}
    $$\int_{\{u<v\}}(dd^cv)^q\wedge T\leq \int_{\{u<v\}\cup\{u=v=-\infty\}}(dd^cu)^q\wedge T.$$

\section{The classes $\mathcal{E}_{p}^{T}(\Omega)$ and $\mathcal{F}_{p}^{T}(\Omega)$}
\subsection{Preliminary results}
Let $\Omega$ be a hyperconvex domain of $\Bbb C^n$, that means it is open, bounded, connected and that there exists  $h\in PSH^-(\Omega)$ such that for all $c<0$, $\{z\in \Omega,\ h(z)<c\}$ is relatively compact in $\Omega$ where $PSH^-(\Omega)$  is the set of negative psh functions.
Let us introduce the Cegrell pluricomplex class $\mathcal{E}_{0}^{T}(\Omega)$ associated to $T$, slightly different to a class introduced in \cite{Ha-Du}, as follows:
$$ \mathcal{E}_{0}^{T}(\Omega):=\left\{\varphi \in PSH^-(\Omega)\cap L^{\infty}(\Omega);\ \lim_{z\to \partial \Omega\cap Supp\; T}\varphi(z)=0,\ \int_{\Omega}(dd^c\varphi)^q\wedge T<+\infty\right\}.$$
Using the same proof as in \cite{Ha-Du}, we can prove easly that this class is a convex cone and that for all $\psi \in PSH^-(\Omega)$ and $\varphi \in \mathcal{E}_{0}^{T}(\Omega)$ one has  $\max(\varphi,\psi)\in \mathcal{E}_{0}^{T}(\Omega)$.\\
In this section we introduce new energy classes $\mathcal{E}_{p}^{T}(\Omega)$ and $\mathcal{F}_{p}^{T}(\Omega)$, similar to Cegrell's ones and we will show that the Monge-Amp\`ere operator is well defined on them.\\
\begin{defn}\label{defn1}
For every real $p\geq 1$ we define $\mathcal{E}_{p}^{T}(\Omega)$ as the set:
$$\mathcal{E}_{p}^{T}(\Omega):=\left\{\varphi \in PSH^-(\Omega); \ \exists\ \mathcal{E}_{0}^{T}(\Omega)\ni \varphi_j\searrow \varphi, \ \sup_{j\geq1}\int_{\Omega}(-\varphi_j)^p(dd^c\varphi_j)^q\wedge T<+\infty\right\}.$$
When the sequence $(\varphi_j)_j$ associated to $\varphi$ can be chosen such that $$\ds\sup_{j\geq1}\int_{\Omega}(dd^c\varphi_j)^q\wedge T<+\infty,$$ we say that $\varphi \in \mathcal{F}_{p}^{T}(\Omega)$.
\end{defn}
It's Easy to check that $\mathcal{E}_{0}^{T}(\Omega)\subset\mathcal{F}_{p}^{T}(\Omega)\subset\mathcal{E}_{p}^{T}(\Omega)$ and that, using H\"older's Inequality, one has  $\mathcal{F}_{p_1}^{T}(\Omega)\subset\mathcal{F}_{p_2}^{T}(\Omega)$ for all $p_2\leq p_1$.\\
We recall the following result which will be  useful to prove some properties of our classes.
\begin{theo}\label{th 1}(See \cite{Da-El})
Suppose that $u,v\in\mathcal{E}_{0}^{T}(\Omega)$. If $p\geq 1$ then for every $0\leq s\leq q$ one has
$$\begin{array}{l}
 \ds\int_{\Omega}(-u)^p (dd^c u)^s \wedge (dd^c v)^{q-s}\wedge T\\
\ds\leq D_{s,p}\left( \int_{\Omega}(-u)^p (dd^c u)^q \wedge T\right)^{\frac{p+s}{p+q}}\left( \int_{\Omega}(-v)^p (dd^c v)^q \wedge T\right)^{\frac{q-s}{p+q}}
\end{array}
$$
where $D_{s,1}=e^{(j+1)(q-j)}$ and $D_{s,p}=p^{\frac{(p+s)(q-s)}{p-1}}$, $p>1$.
\end{theo}

We begin by showing that the two introduced classes inherit some properties of the energy class $\mathcal{E}_{0}^{T}(\Omega).$

\begin{theo}
The classes $\mathcal{E}_{p}^{T}(\Omega)$ and $\mathcal{F}_{p}^{T}(\Omega)$ are convex cones.
\end{theo}
\begin{proof}
It suffices to prove that $u+v\in \mathcal{E}_{p}^{T}(\Omega)$ for every $u,v\in \mathcal{E}_{p}^{T}(\Omega)$. Let $(u_j)_j$ and $(v_j)_j$ be two sequences that decrease to $u$ and $v$ respectively as in Definition \ref{defn1}. We want to estimate $$\int_\Omega (-u_j-v_j)^p (dd^c(u_j+v_j))^q\wedge T.$$
Thanks to Minkowsky Inequality, it is enough to estimate the following terms:
$$\int_{\Omega}(-u_j)^p (dd^c u_j)^s \wedge(dd^cv_j)^{q-s} \wedge T$$ and $$\int_{\Omega}(-v_j)^p (dd^c u_j)^s \wedge(dd^cv_j)^{q-s} \wedge T$$
for all $0<s<q$.
Using Theorem \ref{th 1}, we can estimate last terms by
$$\int_{\Omega}(-u_j)^p (dd^c u_j)^q\wedge T\quad \hbox{and} \quad \int_{\Omega}(-v_j)^p (dd^c v_j)^q\wedge T.$$
As these sequences are uniformly bounded by the definition of $\mathcal{E}_{p}^{T}(\Omega)$, the result follows.
\end{proof}
\begin{prop}\label{prop 1}
Let $u\in\mathcal{E}_{p}^{T}(\Omega)$  (resp. $\mathcal{F}_{p}^{T}(\Omega)$) and $v\in PSH^-(\Omega)$. Then the function  $w:=\max(u,v)$ is in $\mathcal{E}_{p}^{T}(\Omega)$ (resp. in $\mathcal{F}_{p}^{T}(\Omega)$).
\end{prop}
\begin{proof}
Let $(u_j)_j$ be a sequence that decreases to $u$ as in Definition \ref{defn1} and take $w_j:=\max(u_j,v)$. The sequence $(w_j)$ decreases to $w$. So it's enough to prove that $$\ds\sup_j\int_{\Omega}(-w_j)^p (dd^cw_j)^q\wedge T<+\infty.$$ Thanks to  Theorem \ref{th 1}, one has
$$\begin{array}{lcl}
 \ds\int_{\Omega}(-w_j)^p (dd^c w_j)^q\wedge T&\leq&\ds \int_{\Omega}(-u_j)^p (dd^c w_j)^q\wedge T\\
&\leq&\ds D_{0,p}\left( \int_{\Omega}(-u_j)^p (dd^c u_j)^q \wedge T\right)^{\frac{p}{p+q}}\left( \int_{\Omega}(-w_j)^p (dd^c w_j)^q \wedge T\right)^{\frac{q}{p+q}}.
\end{array}
$$
Therefore $$\int_{\Omega}(-w_j)^p (dd^c w_j)^q\wedge T\leq D_{0,p}^{\frac{p+q}{p}}\int_{\Omega}(-u_j)^p (dd^c u_j)^q\wedge T .$$
The right-hand side is uniformly bounded because $u\in\mathcal{E}_{p}^{T}(\Omega)$ and the result follows.
\end{proof}
The most important result of this section is the following theorem which proves that the  Monge-Amp\`ere operator $(dd^c\centerdot)^q\wedge T$ is well defined on the new classes.
\begin{theo}\label{th 3}
Let $u\in\mathcal E_p^T(\Omega)$ and $(u_j)_j$ be a sequence of psh functions that decreases to $u$ as in Definition \ref{defn1}.  Then $(dd^cu_j)^q\wedge T$ converges weakly to a positive measure $\mu$ and this limit is independent of the choice of the sequence $(u_j)_j$. We set $(dd^cu)^q\wedge T:=\mu$.
\end{theo}
\begin{proof}
Let $0\leq\chi\in\mathcal D(\Omega)$, $\delta=\sup\{u_1(z);\ z\in Supp\chi\}$ and $\varepsilon>0$. There exists a sequence $(r_j)_j$ such that $0<r_j<r_{j-1}$ and $$r_j<dist(\{u_j<\frac{\delta}{2}\},\Omega^c).$$
Let $$u_{r_j}(z):=\int_{\Bbb B} u_j(z+r_j\xi)dV(\xi)$$ where $dV$ is the normalized Lebesgue measure on the unit ball $\Bbb B$. Then one has
$$\left|\int_\Omega\chi (dd^cu_{r_j})^q \wedge T-\chi (dd^cu_j)^q\wedge T\right|<\varepsilon.$$
The function $u_{r_j}$ is continuous, psh on $\{u_j<\frac{\delta}{2}\}$ and $u_j\leq u_{r_j}$ on $\Omega$. Let $\widetilde{u_j}=\max(u_{r_j}+\delta,2u_j)$. Then the sequence $(\widetilde{u_j})_j$ decreases to a psh function  $\widetilde{u}$ and $\widetilde{u_j}\in \mathcal{E}_{0}^{T}(\Omega)$ by Proposition \ref{prop 1}. Furthermore, using the same technic of the previous proof, we obtain
$$\sup_{j\geq1}\int_{\Omega}(-\widetilde{u_j})^p(dd^c\widetilde{u_j})^q\wedge T<+\infty.$$
The proof of the theorem will be complete if we show that $$\ds \lim_{j\to +\infty}\int_\Omega \chi (dd^c\widetilde{u_j})^q \wedge T$$  exists.\\
Let $h$ be an exhaustion function in $\mathcal E_0^T(\Omega)$. Then
$$\begin{array}{l}
 \ds\int_\Omega(-\widetilde{u})^p(dd^ch)^q\wedge T=\ds \lim_{j\to+\infty}\int_\Omega(-\widetilde{u_j})^p(dd^ch)^q\wedge T\\
\leq \ds D_{0,p}\sup_{j\geq 1}\left( \int_{\Omega}(-\widetilde{u_j})^p (dd^c \widetilde{u_j})^q \wedge T\right)^{\frac{p}{p+q}}\left( \int_{\Omega}(-h)^p (dd^c h)^q \wedge T\right)^{\frac{q}{p+q}}<+\infty.
\end{array}$$
Thanks to Dabbek-Elkhadhra \cite{Da-El},  the sequence of measures $(dd^c\max(\widetilde{u_j},-k))^q\wedge T$ converges weakly for every $k$. So it is enough to control $$\left|\int\chi (dd^cu_{r_j})^q \wedge T-\chi(dd^c\max(\widetilde{u_j},-k))^q\wedge T \right|.$$
Since $\widetilde{u_j}$ is continuous near $Supp\chi$ then
$$\begin{array}{ll}
 &\ds\left|\int\chi (dd^cu_j)^q \wedge T-\chi(dd^c\max(\widetilde{u_j},-k))^q\wedge T \right|\\
=&\ds\left|\int_{\{\widetilde{u}\leq -k\}}\chi (dd^c\widetilde{u_j})^q \wedge T +\int_{\{\widetilde{u}> -k\}}\chi (dd^c\widetilde{u_j})^q \wedge T\right.\\
&\ds\left.-\int_{\{\widetilde{u}\leq -k\}}\chi (dd^c\max(\widetilde{u_j},-k))^q \wedge T-\int_{\{\widetilde{u}> -k\}}\chi (dd^c\max(\widetilde{u_j},-k))^q \wedge T \right|\\
\leq&\ds\int_{\{\widetilde{u}\leq -k\}}\chi (dd^c\widetilde{u_j})^q \wedge T+\int_{\{\widetilde{u}\leq -k\}}\chi (dd^c\max(\widetilde{u_j},-k))^q \wedge T\\
\leq& \ds\frac{\sup\chi}{k^p}\int_{\{-\widetilde{u}\geq k\}}k^p\left[(dd^c\widetilde{u_j})^q \wedge T+(dd^c\max(\widetilde{u_j},-k))^q \wedge T\right]\\
\leq& \ds\frac{\sup\chi}{k^p} \int_{\Omega}(-\widetilde{u})^p(dd^c\widetilde{u_j})^q\wedge T+(-\max(\widetilde{u_j},-k))^pdd^c\max(\widetilde{u_j},-k))^q\wedge T\\
\leq &\ds C\frac{\sup\chi}{k^p} \sup_{m\geq 1}\int_{\Omega}(-\widetilde{u}_m)^p(dd^c\widetilde{u_m})^q\wedge T.
\end{array}
$$
This completes the proof of the theorem.
\end{proof}

\begin{theo}
If $u\in \mathcal{E}_1^{T}(\Omega)$ then $$\int_{\Omega} u (dd^cu)^q\wedge T>-\infty.$$ Moreover, if $v_j\in PSH^-(\Omega)$ such that $(v_j)_j$ decreases to $u$ then $$ \int_{\Omega} v_j (dd^cv_j)^q\wedge T\hbox{  converges to }\int_{\Omega} u (dd^cu)^q\wedge T .$$
\end{theo}
\begin{proof}
 Since $u\in \mathcal{E}_1^{T}(\Omega)$ then  there exists a sequence $(u_j)_j\subset \mathcal{E}_0^{T}$ such that $$ \lim_{j\to +\infty}u_j=u\ \  and \ \ \alpha:=\sup_j \int -u_j (dd^c u_j)^q\wedge T<+\infty.$$

Let us prove that
$$\lim_{j\rightarrow +\infty}\int_\Omega u_j(dd^cu_j)^q\wedge T=\int_\Omega u(dd^cu)^q\wedge T.$$
For every $k\geq j$ and $\varepsilon >0$, one has
$$\begin{array}{l}
 \ds\int_{\Omega} -u_j(dd^cu_j)^q\wedge T\\
 \leq\ds \int_{\Omega} -u_j(dd^cu_k)^q\wedge T\\
=\ds\int_{\{u_j\geq -\varepsilon\}} -u_j(dd^cu_k)^q\wedge T +\int_{\{u_j< -\varepsilon\}}-u_j(dd^cu_k)^q\wedge T
\end{array}$$
and $$\begin{array}{l}
 \ds\int_{\{u_j\geq -\varepsilon\}} -u_j(dd^cu_k)^q\wedge T \\
=\ds\int_{\{u_j\geq -\varepsilon\}} -\max(u_j,-\varepsilon)(dd^cu_k)^q\wedge T\\
\leq\ds \left(\int_\Omega -\max(u_j,-\varepsilon)(dd^c\max(u_j,-\varepsilon))^q\wedge T\right)^{\frac1{q+1}}\left(\int_\Omega -u_k(dd^cu_k)^q\wedge T\right)^{\frac q{q+1}}\\
\leq \ds \left(\varepsilon \int_\Omega (dd^cu_j)^q\wedge T\right)^{\frac 1{q+1}}\alpha^{\frac q{q+1}}
\end{array}
$$
This goes to $0$ when $\varepsilon\rightarrow 0$. By Theorem \ref{th 3} we obtain
$$\limsup_{k\to +\infty}\int_{\{u_j< -\varepsilon\}} -u_j(dd^cu_k)^q\wedge T\leq \int_{\Omega}-u_j (dd^cu)^q\wedge T .$$
Now since $-u_j$ is lower semi-continuous then
$$\liminf_{k\to +\infty}\int_{\Omega} -u_j(dd^cu_k)^q\wedge T\geq \int_{\Omega}-u_j (dd^cu)^q\wedge T .$$
Hence for all $j$,
$$\lim_{k\to +\infty}\int_{\Omega} u_j(dd^cu_k)^q\wedge T= \int_{\Omega} u_j(dd^cu)^q\wedge T.$$
It follows that
$$\begin{array}{l}
 \ds\lim_{j\to +\infty}\int_{\Omega} u_j(dd^cu_j)^q\wedge T\\
\geq \ds\lim_{j\to +\infty} \lim_{k\to+\infty}\int_{\Omega} u_j(dd^cu_k)^q\wedge T = \int_{\Omega} u(dd^cu)^q\wedge T\\
\geq \ds\limsup_{k\to +\infty}\int_{\Omega} u(dd^cu_k)^q\wedge T = \limsup_{k\to +\infty}\lim_{j\to +\infty}\int_{\Omega} u_j(dd^cu_k)^q\wedge T\\
\geq \ds\lim_{j\to +\infty}\int_{\Omega} u_j(dd^cu_j)^q\wedge T.
\end{array}
 $$
Thus
\begin{equation}\label{eq 2.1}
    \lim_{j\rightarrow +\infty}\int_{\Omega} u_j(dd^cu_j)^q\wedge T= \int_{\Omega}  u(dd^cu)^q\wedge T.
\end{equation}
As $(v_k)_k$ decreases to $u$ then $v_k \in \mathcal{E}_1^T(\Omega)$. It follows that
        \begin{equation}\label{eq 2.2}
            \int_{\Omega} \max(u_j,v_k)(dd^c\max(u_j,v_k))^q\wedge T\geq \int_{\Omega} u_j (dd^cu_j)^q\wedge T\geq -\alpha.
        \end{equation}
        Moreover, $(\max(u_j,v_k))_{j\in\Bbb N}\subset\mathcal E_0^T(\Omega)$ and decreases to $v_k$ so thanks to Equality (\ref{eq 2.1}),
        \begin{equation}\label{eq 2.3}
            \lim_{j\to+\infty}\int_{\Omega} \max(u_j,v_k)(dd^c\max(u_j,v_k))^q\wedge T =\int_{\Omega}  v_k(dd^cv_k)^q\wedge T.
        \end{equation}
        By tending $j\to+\infty$, Inequality (\ref{eq 2.2}), Equalities (\ref{eq 2.1}) and (\ref{eq 2.3}) give
        $$\int_{\Omega}  v_k(dd^cv_k)^q\wedge T\geq \int_{\Omega}  u(dd^cu)^q\wedge T.$$
        Thus
        \begin{equation}\label{eq 2.4}
            \liminf_{k\to+\infty}\int_{\Omega}  v_k(dd^cv_k)^q\wedge T\geq \int_{\Omega}  u(dd^cu)^q\wedge T.
        \end{equation}
        With the same reason, as $(\max(u_j,v_k))_{k\in\Bbb N}$ decreases to $u_j$ then
        $$\int_{\Omega}  u_j(dd^cu_j)^q\wedge T\geq \limsup_{k\to+\infty}\int_{\Omega}  v_k(dd^cv_k)^q\wedge T.$$
        Hence
        \begin{equation}\label{eq 2.5}
            \limsup_{k\to+\infty}\int_{\Omega}  v_k(dd^cv_k)^q\wedge T\leq \int_{\Omega}  u(dd^cu)^q\wedge T.
        \end{equation}
        The result follows from Inequalities (\ref{eq 2.4}) and (\ref{eq 2.5}).
    \end{proof}
    \begin{rem}
        Claim that if $u\in\mathcal E_1^T(\Omega)$ and  $(u_j)_j$ is a decreasing sequence  to $u$ as in Definition \ref{defn1} then
        $$ \int_{\Omega} u_j (dd^cu_j)^q\wedge T\hbox{  decreases to }\int_{\Omega} u (dd^cu)^q\wedge T .$$
    \end{rem}
\subsection{Comparaison theorems}
        We recall two classes $\mathcal{E}^{T}(\Omega)$ and $\mathcal{F}^{T}(\Omega)$ introduced in \cite{Ha-Du} where authors prove that  the Monge-Amp\`ere operator $(dd^c\centerdot)^q\wedge T$ is well defined on them.
    \begin{defn}\label{defn2}
        We say that $u\in \mathcal{F}^{T}(\Omega)$ if there exists a sequence $(u_j)_j\subset\mathcal{E}_0^{T}(\Omega)$ which decreases to $u$ such that $$\sup_j\int_\Omega (dd^cu_j)^q\wedge T< +\infty.$$
        A function $u$ will belong to $ \mathcal{E}^{T}(\Omega)$  if for all $z\in \Omega$ there exist a neighborhood $\omega$ of $z$ and a function   $v\in\mathcal{F}^{T}(\Omega)$  such that $u=v$ on $\omega$.
    \end{defn}
        As a consequence, for every $p\geq 1$ one has $\mathcal F_p^T(\Omega)\subset \mathcal F^T(\Omega)\subset\mathcal E^T(\Omega)$ but we dont know any relationship between $\mathcal E_p^T(\Omega)$ and $\mathcal E^T(\Omega)$.

    \begin{lem}
        Let $u,v\in PSH(\Omega)\cap L^{\infty}(\Omega)$ and $U$ be an open subset of $\Omega$ such that $u=v$ near $\partial U$. Then $$\int_U(dd^cu)^q\wedge T=\int_U(dd^cv)^q\wedge T$$
    \end{lem}
    \begin{proof}
        Let $u_{\varepsilon}$ and $v_{\varepsilon}$ be  the usual regularization of $u$ and $v$ respectively. Choose $U'\subset\subset U$ such that $u=v$ near $\partial U'$. If $\varepsilon>0$ is small enough, one has $u_{\varepsilon}=v_{\varepsilon}$ near $\partial U'$ and if we take $\chi\in \mathcal{D}(U')$ with $\chi=1$ near $\{u_{\varepsilon}\neq v_{\varepsilon}\}$ then $dd^c\chi=0$ on $\{u_{\varepsilon}\neq v_{\varepsilon}\}$. So
        $$\begin{array}{lcl}
            \ds\int_{\Omega}\chi (dd^cu_{\varepsilon})^q \wedge T&=&\ds\int_{\Omega} u_{\varepsilon} dd^c\chi\wedge (dd^cu_{\varepsilon})^{q-1} \wedge T\\
	       &=&\ds\int_{\Omega} v_{\varepsilon} dd^c\chi\wedge (dd^cu_{\varepsilon})^{q-1} \wedge T\\
	       &=&\ds\int_{\Omega}\chi (dd^cv_{\varepsilon})^q \wedge T.
        \end{array}$$
        Hence $$\int_{\Omega}\chi (dd^cu)^q \wedge T=\int_{\Omega}\chi (dd^cv)^q \wedge T.$$
        The result follows.
    \end{proof}
    \begin{cor}\label{cor 1}
        Let  $u,v\in\mathcal{F}^{T}(\Omega)$. Assume that there exists an open subset $U$ of $\Omega$ such that $u=v$ near $\partial U$. Then $$\int_U(dd^cu)^q\wedge T=\int_U(dd^cv)^q\wedge T.$$
    \end{cor}
    \begin{proof}
        Let $u,v\in\mathcal F^{T}(\Omega)$ and $w\in \mathcal{E}_0^{T}(\Omega)$ such that $w(z) \neq0$ for all $z$. Then $u_j:=\max(u,jw)$ and $v_j=\max(v,jw)$ belong to $\mathcal{E}_0^{T}(\Omega)$ and they are equal on $\partial U$. The result follows from the previous lemma.
    \end{proof}
        Now we recall a result due to \cite{Ha-Du} and we give a different proof.
    \begin{prop}\label{prop 2}(See \cite{Ha-Du})
        For  $u,v\in\mathcal{F}^{T}(\Omega)$ such that $u\leq v$ on $\Omega$ one has
        $$\int_{\Omega}(dd^cv)^q\wedge T\leq\int_{\Omega}(dd^cu)^q\wedge T.$$
    \end{prop}
    \begin{proof}
        Let $(u_j)_j$ and $(v_j)_j$ be the corresponding decreasing sequences to $u$ and $v$ respectively as in Definition \ref{defn2}. Replace $v_j$ by $\max(u_j,v_j)$ we can assume that $u_j\leq v_j$ for all $j\in\Bbb N$. For $h\in \mathcal E_0^T(\Omega)$ and $\varepsilon>0$ we have
        $$\begin{array}{lcl}
            \ds\int_{\Omega}-h(dd^cv_j)^q\wedge T&\leq&\ds\int_{\Omega}-h(dd^cu_j)^q\wedge T\\
            &\leq &\ds \int_{\Omega}-h(dd^cu)^q\wedge T+\limsup_{j\to +\infty}\int_{\{h>-\varepsilon\}}-h(dd^cu_j)^q\wedge T\\
            &\leq&\ds \int_{\Omega}-h(dd^cu)^q\wedge T+\varepsilon\limsup_{j\to +\infty}\int_{\Omega}(dd^cu_j)^q\wedge T.
        \end{array}$$
        By tending $\varepsilon$ to $0$ we obtain $$\int_{\Omega}-h(dd^cv)^q\wedge T\leq \int_{\Omega}-h(dd^cu)^q\wedge T$$
        The result follows by choosing $h$ decreases to $-1$.
    \end{proof}
    \begin{lem}\label{lem 2}
        Let $u\in\mathcal{F}^{T}(\Omega)$ then there exists a sequence  $ (u_j)_j\subset \mathcal{E}_{0}^{T}(\Omega)\cap {\mathcal{C}(\overline{\Omega})}$ that decreases to $u$.
    \end{lem}
        We claim that this lemma was cited in \cite[th.5.1]{Ha-Du} with uncompleted proof; in fact authors had used a comparaison theorem, proved by Dabbek-Elkhadhra \cite{Da-El} only for bounded psh functions,  in $\mathcal F^T(\Omega)$ where  functions are not in general bounded.
    \begin{proof}
        We refer to Cegrell \cite[Th.2.1]{Ce2} for the construction of the sequence $(u_j)_j$. It remains to show that $$\ds \int_{\Omega}(dd^cu_j)^q\wedge T <\infty .$$
        As $u_j\geq u$ then by Proposition \ref{prop 2} one has
        $$\int_{\Omega}(dd^cu_j)^q\wedge T \leq \int_{\Omega}(dd^cu)^q\wedge T<+\infty.$$
    \end{proof}
\section{$C_T$-quasicontinuity}
        Now we establish the quasicontinuity of psh functions belong to  $\mathcal{F}^{T}(\Omega)$ and $\mathcal E_p^{T}(\Omega)$. We need to recall some notions given in \cite{Da-El} (see also \cite{Ko}) about the capacity associated to $T$ which is defined as
        $$C_T(K,\Omega)=C_T(K)=\sup \left\{\int_K(dd^cv)^q\wedge T,\  v\in PSH(\Omega,[-1,0])\right\}.$$
        for all compact subset $K$ of $\Omega$. If $E$  is a subset of $\Omega$, we define
        $$C_T(E,\Omega)=\sup \{C_T(K),\ K\hbox{ compact subset of }E\}.$$  We refer to \cite{Da-El,Ko} for the properties of this capacity.
    \begin{defn}\
        \begin{itemize}
            \item A subset $A$ of $\Omega$ is said to be $T$-pluripolar if $C_T(A,\Omega)=0$.
            \item A psh  function $u$ is said to be quasicontinuous with respect to $C_T$, if for every $\varepsilon>0$, there exists an open subset $O_{\varepsilon}$ such that $C_T(O_{\varepsilon},\Omega)<\varepsilon$ and $u$ is continuous on $\Omega\smallsetminus O_{\varepsilon}$.
        \end{itemize}
    \end{defn}
    \begin{prop}\label{prop 3}
        Let $u\in\mathcal F^T(\Omega)$. Then for every $s>0$ one has
        $$s^q C_T(\{u\leq -s\},\Omega)\leq \int_\Omega (dd^cu)^q\wedge T.$$
        In particular, the  set $\{u=-\infty\}$ is $T$-pluripolar.
    \end{prop}
    \begin{proof}
        Let $(u_j)_j\subset \mathcal E_0^T(\Omega)$ be a decreasing sequence to $u$ on $\Omega$ as in Definition \ref{defn2}. Take $s>0$, $v\in PSH(\Omega,[-1,0])$ and $K$ a compact subset in $\{u_j\leq -s\}$. Thanks to the comparaison principle (for bounded psh functions), we have
        $$\begin{array}{lcl}
            \ds\int_K (dd^cv)^q\wedge T&\leq&\ds \int_{\{s^{-1}u_j<v\}}(dd^cv)^q\wedge T\leq \frac1{s^q}\int_{\{s^{-1}u_j<v\}}(dd^cu_j)^q\wedge T\\
            &\leq &\ds\frac1{s^q}\int_\Omega(dd^cu_j)^q\wedge T
        \end{array}$$
        It follows that
        $$C_T(\{u_j\leq -s\},\Omega)\leq \frac1{s^q}\int_\Omega(dd^cu_j)^q\wedge T.$$
        By tending $j$ to infinity, we obtain
        $$C_T(\{u\leq -s\},\Omega)\leq \frac1{s^q}\int_\Omega(dd^cu)^q\wedge T.$$
    \end{proof}
    \begin{cor}\label{cor 2}
        Every $u\in\mathcal{F}^{T}(\Omega)$  is $C_T$-quasicontinuous.
    \end{cor}
    \begin{proof}
        Let $u\in\mathcal{F}^{T}(\Omega)$ and $\varepsilon>0$. Denote by $B_u(t):=\{z\in\Omega;\ u(z)<t\},\ t\leq 0$. By Proposition \ref{prop 3}, there is $s_\varepsilon\geq1$ such that $C_T(B_u(-s_\varepsilon),\Omega)<\frac\varepsilon2$. The function $u_\varepsilon:=\max(u,-s_\varepsilon)$ is bounded on $\Omega$ so thanks to Dabbek-Elkhadhra \cite{Da-El}, there is an open subset $\mathcal O$ in $\Omega$ such that $C_T(\mathcal O,\Omega)<\frac\varepsilon2$ and $u_\varepsilon$ is continuous on  $\Omega\smallsetminus \mathcal O$. The result follows by taking $\mathcal O_\varepsilon=\mathcal O\cup B_u(-s_\varepsilon)$.
    \end{proof}
        To study the $C_T$-quasicontinuity on $\mathcal E_p^T(\Omega)$, we will proceed as in the previous case.
    \begin{prop}\label{prop 4}
        Let $u\in\mathcal E_p^T(\Omega)$ and $(u_j)_j\subset \mathcal E_0^T(\Omega)$ decreases to $u$ on $\Omega$ as in Definition \ref{defn1}. Then for every $s>0$ one has
        $$s^{p+q} C_T(\{u\leq -2s\},\Omega)\leq \sup_{j\geq1}\int_\Omega (-u_j)^p(dd^cu_j)^q\wedge T.$$
        In particular, the  set $\{u=-\infty\}$ is $T$-pluripolar.
    \end{prop}
    \begin{proof}
        Let  $s>0$, $v\in PSH(\Omega,[-1,0])$. Thanks to comparaison principle (for bounded psh functions), we have
        $$\begin{array}{lcl}
            \ds\int_{\{u_j\leq -2s\}} (dd^cv)^q\wedge T&\leq&\ds \int_{\{u_j<-s+sv\}}(dd^cv)^q\wedge T\leq \frac1{s^q}\int_{\{s^{-1}u_j<-1+v\}}(dd^cu_j)^q\wedge T\\
            &\leq &\ds\frac1{s^{p+q}}\int_\Omega(-u_j)^p(dd^cu_j)^q\wedge T
        \end{array}$$
        It follows that
        $$C_T(\{u_j\leq -2s\},\Omega)\leq \frac1{s^{p+q}}\sup_{m\geq 1}\int_\Omega(-u_m)^p(dd^cu_m)^q\wedge T.$$
        By tending $j$ to infinity, we obtain
        $$C_T(\{u\leq -2s\},\Omega)\leq \frac1{s^{p+q}}\sup_{m\geq 1}\int_\Omega(-u_m)^p(dd^cu_m)^q\wedge T.$$
    \end{proof}
        By the same argument as in corollary \ref{cor 2}  we can easily deduce the following result:
    \begin{cor}
        Every function in $\mathcal E_p^T(\Omega)$ is $C_T$-quasicontinuous.
    \end{cor}
        Now we need a first version of the comparaison principle where one of the functions will be unbounded. This result was proved in  \cite{Da-El} for bounded functions.
    \begin{theo}\label{th 5}
        Let $u\in \mathcal{F}^{T}(\Omega)$ and $v\in PSH(\Omega)\cap L^\infty(\Omega)$ such that
        $$\liminf_{z\to \partial \Omega\cap SuppT} u(z)-v(z)\geq 0.$$
        Then
        $$\int_{\{u<v\}}(dd^cv)^q\wedge T\leq \int_{\{u<v\}}(dd^cu)^q\wedge T.$$
    \end{theo}
    \begin{proof}
        Firstly we assume that $u$ and $v$ are continuous on a neighborhood $W$ of $Supp T$. Without loss of generality we can assume that $u<v$ on $W$ and $u=v$ on $\partial W$. Let $v_{\varepsilon}:=\max(u,v-\varepsilon)$ then one has $v_{\varepsilon}=u$ on $\partial W$ and
        $$\int_{\{u<v\}}(dd^cv_{\varepsilon})^q\wedge T= \int_{\{u<v\}}(dd^cu)^q\wedge T.$$
        Since the family of measures  $(dd^cv_{\varepsilon})^q\wedge T$ converges weakly to $(dd^cu)^q\wedge T$ as $\varepsilon\to 0$, then we obtain
        $$\int_{\{u<v\}}(dd^cv)^q\wedge T= \int_{\{u<v\}}(dd^cu)^q\wedge T.$$
        Let now treat the general cas. Replace $u$ by $u+\delta$ if necessary, we can assume that $\liminf(u-v)\geq 2\delta$; so there is an open subset $\mathcal{O}\subset \subset \Omega$ such that $u(z)\geq v(z)+\delta$ for all $z\in \Omega\smallsetminus \mathcal{O}$. Let $(u_k)_k$ and $(v_j)_j$ be two smooth sequences of psh functions which decrease respectively to $u$ and $v$ on a neighborhood of $\overline{\mathcal{O}}$ such that $u_k\geq v_j$ on $\partial \overline{\mathcal{O}}\cap Supp T$ for $j\geq k$. Using the previous argument we obtain
        $$\int_{\{u_k<v_j\}}(dd^cv_j)^q\wedge T= \int_{\{u_k<v_j\}}(dd^cu_k)^q\wedge T.$$
        For $\varepsilon > 0$, there exists an open subset $G$ of $\Omega$ such that $C_T(G,\Omega)<\varepsilon$ and $u,v$ are continuous on $\Omega\smallsetminus G$. We can write $v=\varphi + \psi$ where $\varphi$ is continuous on $\Omega$ and $\psi=0$ on $\Omega \smallsetminus G.$ Take $U:=\{u_k<\varphi\}$ then
        $$\int_U(dd^cv)^q\wedge T\leq \lim_{j\to +\infty}\int_U(dd^cv_j)^q\wedge T.$$
        Since $U\cup G=\{u_k<v\}\cup G$ then
        $$\begin{array}{l}
            \ds\int_{\{u_k<v\}}(dd^cv)^q\wedge T\\
            \ds\leq \int_U(dd^cv)^q\wedge T+\int_G(dd^cv)^q\wedge T\\
            \ds\leq \lim_{j\to +\infty}\int_U(dd^cv_j)^q\wedge T+\int_G(dd^cv)^q\wedge T\\
            \ds\leq \lim_{j\to +\infty}\left(\int_{\{u_k<v_j\}}(dd^cv_j)^q\wedge T+\int_G(dd^cv_j)^q\wedge T\right)+ \int_G(dd^cv)^q\wedge T\\
            \ds \leq \lim_{j\to +\infty}\int_{\{u_k<v_j\}}(dd^cv_j)^q\wedge T+2\varepsilon||v||_{\infty}^q\\
            \ds \leq \lim_{j\to +\infty}\int_{\{u_k<v_j\}}(dd^cu_k)^q\wedge T+2\varepsilon||v||_{\infty}^q.
        \end{array}$$
        Now as $\{u_k<v_j\}\downarrow{\{u_k\leq v\}}$, ${\{u_k<v\}}\uparrow {\{u<v\}}$ then
        $$\int_{\{u<v\}}(dd^cv)^q\wedge T\leq \lim_{k\to +\infty}\int_{\{u_k\leq v\}}(dd^cu_k)^q\wedge T+2\varepsilon||v||_{\infty}^q.$$
        The continuity of $u$ and $v$ on $\Omega\smallsetminus G$ gives that ${\{u\leq v\}}\smallsetminus G$ is a closed subset of $\Omega$. It follows that
        $$\int_{\{u\leq v\}\smallsetminus G}(dd^cu)^q\wedge T\geq \lim_{k\to +\infty}\int_{\{u\leq v\}\smallsetminus G}(dd^cu_k)^q\wedge T.$$
        Thus
        $$\begin{array}{lcl}
            \ds\int_{\{u\leq v\}}(dd^cu)^q\wedge T&\geq&\ds \int_{\{u\leq v\}\smallsetminus G}(dd^cu)^q\wedge T\\
            &\geq &\ds \lim_{k\to +\infty}\int_{\{u\leq v\}\smallsetminus G}(dd^cu_k)^q\wedge T\\
            &\geq&\ds  \lim_{k\to +\infty}\left(\int_{\{u_k<v\}}(dd^cu_k)^q\wedge T-\int_G(dd^cu_k)^q\wedge T\right)\\
            &\geq&\ds \lim_{k\to +\infty}\int_{\{u_k<v\}}(dd^cu_k)^q\wedge T-\varepsilon||v||_{\infty}^q.
        \end{array}$$
        So
        $$\int_{\{u<v\}}(dd^cv)^q\wedge T\leq\int_{\{u\leq v\}}(dd^cu)^q\wedge T+3\varepsilon||v||_{\infty}^q.$$
        By tending $\varepsilon$ to $0$, we obtain
        $$\int_{\{u<v\}}(dd^cv)^q\wedge T\leq\int_{\{u\leq v\}}(dd^cu)^q\wedge T$$
        As $\{u+\rho<v\}\uparrow \{u<v\}$ and $\{u+\rho\leq v\}\uparrow \{u<v\}$ when $\rho\searrow 0$ then the desired inequality follows by replacing $u$ by $u+\rho$.
    \end{proof}

    Recall that  the Lelong-Demailly number of $T$ with respect to a psh function $\varphi$ is defined as the limit $\nu(T,\varphi):=\lim_{t\to-\infty}\nu(T,\varphi,t)$ where
    $$\nu(T,\varphi,t)=\int_{B_\varphi(t)}T\wedge(dd^c\varphi)^q,\ t<0\ .$$
    The following result was proved in \cite{El} but author has used Stokes formula where  a regularity condition on $\varphi$ is required.
    \begin{theo}\label{th 6}
        Let $\varphi\in\mathcal F^T(\Omega)$ such that $e^\varphi$ is continuous on $\Omega$. Then for every $s,t>0$ one has
        \begin{equation}\label{eq 3.1}
            s^q C_T(B_\varphi(-t-s),\Omega)\leq \nu(T,\varphi,-t)\leq (s+t)^q C_T(B_\varphi(-t),\Omega).
        \end{equation}
        In particular,
        $$\nu(T,\varphi)=\int_{\{\varphi=-\infty\}} T\wedge(dd^c\varphi)^q=\lim_{t\to+\infty} t^q C_T(B_\varphi(-t),\Omega).$$
    \end{theo}

    \begin{proof}
        Let $t,s>0$ and  $v\in PSH(\Omega,[-1,0])$. For $\varepsilon>0$, we set $v_\varepsilon=\max(v,\frac{\varphi+t+\varepsilon}s)$. Thanks to Theorem \ref{th 5} we have
        $$\begin{array}{lcl}
            \ds \int_{B_\varphi(-t-s-\varepsilon)}T\wedge(dd^cv)^q&=& \ds \int_{B_\varphi(-t-s-\varepsilon)}T\wedge(dd^cv_\varepsilon)^q\\
            &\leq&\ds \int_{\{\varphi<-t+sv-\varepsilon\}}T\wedge(dd^cv_\varepsilon)^q\\
            &\leq&\ds \frac1{s^q}\int_{\{\varphi<-t+sv-\varepsilon\}}T\wedge(dd^c\varphi)^q\\
            &\leq&\ds \frac1{s^q}\int_{B_\varphi(-t)}T\wedge(dd^c\varphi)^q.
        \end{array}$$
        By passing to the supremum over all $v\in PSH(\Omega,[-1,0])$, we obtain the following estimate
        $$s^q C_T(B_\varphi(-s-t-\varepsilon),\Omega)\leq \nu(T,\varphi,-t).$$
        By passing to the limit when $\varepsilon\to0$, the left inequality in (\ref{eq 3.1}) is obtained. However, for the right inequality, we remark that the function $\psi=\max(\frac{\varphi}{s+t},-1)$ is psh and satisfies $-1\leq \psi\leq 0$ on $\Omega$, so by Corollary \ref{cor 1}  and using the fact that $\psi>-1$ near $\partial B_\varphi(-t)$ we obtain
        $$\begin{array}{lcl}
            \ds\int_{B_\varphi(-t)}T\wedge(dd^c\varphi)^q&=&\ds(s+t)^q\int_{B_\varphi(-t)}T\wedge(dd^c\psi)^q\\
            &\leq&\ds (s+t)^q C_T(B_\varphi(-t),\Omega)
        \end{array}$$
        and the right inequality in  (\ref{eq 3.1}) follows.\\
        By the right inequality in (\ref{eq 3.1}), we have
        $$\nu(T,\varphi)=\lim_{t\to+\infty}\nu(T,\varphi,-t)\leq \lim_{t\to+\infty} \frac{(s+t)^q}{t^q}t^q C_T(B_\varphi(-t),\Omega)=\lim_{t\to+\infty} t^q C_T(B_\varphi(-t),\Omega).$$
        If we take $\alpha>1$ and $s=\alpha t$ in the left inequality in (\ref{eq 3.1}), we obtain
        $$\begin{array}{lcl}
            \nu(T,\varphi)=\ds\lim_{t\to+\infty}\nu(T,\varphi,-t)&\geq &\ds \lim_{t\to+\infty}\frac{\alpha^q}{(1+\alpha)^q}(1+\alpha)^qt^qC_T(B_\varphi(-(1+\alpha)t),\Omega)\\
            &=&\ds\left(\frac{\alpha}{1+\alpha}\right)^q \lim_{t\to+\infty} t^q C_T(B_\varphi(-t),\Omega).
        \end{array}$$
        The result follows by letting $\alpha\to+\infty$.
    \end{proof}
    \begin{rem}
        Claim that if $\varphi\in \mathcal F_p^T(\Omega)$ where $e^\varphi$ is  continuous on $\Omega$, then thanks to Proposition \ref{prop 4} and Theorem \ref{th 6}, $\nu(T,\varphi)=0.$
    \end{rem}

\section{Main result}
    The aim of this part is to prove the following main result:\\
    \noindent\textbf{Main result (Comparison principle)}
    \emph{Let $u\in \mathcal{F}^{T}(\Omega)$ and $v\in \mathcal E^{T}(\Omega)$. Then}
    $$\int_{\{u<v\}}(dd^cv)^q\wedge T\leq \int_{\{u<v\}\cup\{u=v=-\infty\}}(dd^cu)^q\wedge T.$$
    Before giving the proof, we give some corollaries.
    \subsection{Consequences of the main result}

    \begin{cor}
        Let $u,v\in \mathcal F_p^{T}(\Omega)$ such that $e^u$ is continuous on $\Omega$. Then
        $$\int_{\{u<v\}}(dd^cv)^q\wedge T\leq \int_{\{u<v\}}(dd^cu)^q\wedge T.$$
    \end{cor}
    \begin{proof}
        Thanks to the comparaison principle, we have
        $$\int_{\{u<v\}}(dd^cv)^q\wedge T\leq \int_{\{u<v\}\cup\{u=v=-\infty\}}(dd^cu)^q\wedge T\leq\int_{\{u<v\}}(dd^cu)^q\wedge T+\nu(T,u).$$
        The result follows by the fact that $\nu(T,u)=0$ because $u\in \mathcal F_p^{T}(\Omega)$.
    \end{proof}
    \begin{cor}
        Let $u\in\mathcal{F}^{T}(\Omega)$ and $v\in\mathcal{F}_p^{T}(\Omega)$ such that  $e^v$ is continuous on $\Omega$. We assume that  $$(dd^cu)^q\wedge T\leq (dd^cv)^q\wedge T.$$ Then $C_T(\{u<v\},\Omega)=0$.
    \end{cor}
    \begin{proof}
        Assume that $C_T(\{u<v\},\Omega)>0$, then there exists $\psi\in PSH(\Omega,[0,1])$ such that
        $$\int_{\{u<v\}}(dd^c\psi)^q\wedge T>0.$$ For $\varepsilon>0$ small enough, one has $v+\varepsilon\psi\in\mathcal F^T(\Omega)$ so thanks to the comparaison principle,
        $$\begin{array}{lcl}
            \ds\int_{\{u<v+\varepsilon\psi\}}(dd^c(v+\varepsilon\psi))^q\wedge T&\leq&\ds \int_{\{u<v+\varepsilon\psi\}\cup\{u=v=-\infty\}}(dd^cu)^q\wedge T\\
            &\leq&\ds\int_{\{u<v+\varepsilon\psi\}\cup\{u=v=-\infty\}}(dd^cv)^q\wedge T\\
            &\leq&\ds\int_{\{u<v+\varepsilon\psi\}}(dd^cv)^q\wedge T+\nu(T,v).
        \end{array}$$
        So:
        $$\varepsilon^q\int_{\{u<v\}}(dd^c\psi)^q\wedge T+\int_{\{u<v+\varepsilon\psi\}}(dd^cv)^q\wedge T\leq\int_{\{u<v+\varepsilon\psi\}}(dd^cv)^q\wedge T$$
        which is absurd.
    \end{proof}

\subsection{Proof of the main result}
    To prove the main result, we  shall use a  similar Xing's inequalities (see \cite{Xi1,Xi2} for more details), generalized  to $\mathcal{E}^{T}(\Omega)$. We start by recalling the following lemma:
\begin{lem} \label{lem 3} (See \cite{Ha-Du})
    Let $S$ be a positive closed current of bidimension $(1,1)$ on $\Omega$ and $u,v\in PSH(\Omega)\cap L^{\infty}(\Omega)$. Assume that $u\leq v$ on $\Omega$ and $$\ds\lim_{z\to \partial \Omega}[u(z)-v(z)]=0.$$ Then one has
    $$\int_{\Omega}(v-u)^k dd^cw\wedge S\leq k\int_{\Omega}(1-w)(v-u)^{k-1}dd^cu\wedge S$$
    for all $k\geq 1$ and $w\in PSH(\Omega,[0,1])$.
\end{lem}
\begin{lem}\label{lem 4}
    Let $u,v\in PSH(\Omega)\cap L^{\infty}(\Omega)$ such that $u\leq v$ on $\Omega$ and $$\ds\lim_{z\to \partial \Omega}[u(z)-v(z)]=0.$$ Then  one has
    $$\frac{1}{q!}\int_{\Omega}(v-u)^q dd^cw_1\wedge...\wedge dd^cw_q\wedge T+\int_{\Omega}(r-w_1)(dd^cv)^q\wedge T\leq \int_{\Omega}(r-w_1)(dd^cu)^q\wedge T$$
    for every $r\geq1$ and $w_1,..., w_q\in PSH(\Omega,[0,1])$.
\end{lem}
\begin{proof}
    Let $K\subset \subset \Omega$ and assume that $u=v$ on $\Omega\smallsetminus K$. Using Lemma \ref{lem 3} we obtain
    $$\begin{array}{l}
        \ds\int_{\Omega}(v-u)^q dd^cw_1\wedge...\wedge dd^cw_q\wedge T\\
        \ds\leq q\int_{\Omega}(v-u)^{q-1}dd^cw_1\wedge...\wedge dd^cw_{q-1}\wedge dd^cu\wedge T\\
        \vdots\\
        \ds\leq q!\int_{\Omega}(v-u)dd^cw_1\wedge (dd^cu)^{q-1}\wedge T\\
        \ds\leq q!\int_{\Omega}(w_1-r)dd^c(v-u)\wedge\left(\sum_{i=0}^{q-1}(dd^cu)^i\wedge (dd^cv)^{q-i-1}\right)\wedge T\\
        \ds= q!\int_{\Omega}(r-w_1)dd^c(u-v)\wedge\left(\sum_{i=0}^{q-1}(dd^cu)^i\wedge (dd^cv)^{q-i-1}\right)\wedge T\\
        \ds= q!\int_{\Omega}(r-w_1)((dd^cu)^q-(dd^cv)^q)\wedge T.
    \end{array}$$
    In the general case, for every $\varepsilon>0$ we set $v_{\epsilon}=\max(u,v-\varepsilon)$. Then $v_{\epsilon}\nearrow v$ on $\Omega$ and satisfies $v_{\epsilon}=u$ on $\Omega\smallsetminus K$ for some $K\subset \subset \Omega$. Hence
    $$\frac{1}{q!}\int_{\Omega}(v_{\varepsilon}-u)^q dd^cw_1\wedge...\wedge dd^cw_q\wedge T+\int_{\Omega}(r-w_1)(dd^cv_{\varepsilon})^q\wedge T\leq \int_{\Omega}(r-w_1)(dd^cu)^q\wedge T$$
    Since $v_{\varepsilon}-u\nearrow v-u$,  the family of measures $(dd^cv_{\varepsilon})^q\wedge T$ converges weakly to $(dd^cv)^q\wedge T$ as $\varepsilon \searrow 0$ and the function $r-w_1$ is lower semicontinuous then, by letting $\varepsilon \searrow 0$, we obtain the desired inequality.
\end{proof}
\begin{prop} \label{prop 5} Let $r\geq 1$ and $w\in PSH(\Omega,[0,1])$.
    \begin{enumerate}
        \item[(a)] For every  $u,v\in\mathcal{F}^{T}(\Omega)$ such that $u\leq v$ on $\Omega$ one has
            \begin{equation}\label{eq 4.1}
                 \ds\frac{1}{q!}\int_{\Omega}(v-u)^q  (dd^cw)^q\wedge T+\int_{\Omega}(r-w)(dd^cv)^q \wedge T 	 \leq\ds\int_{\Omega}(r-w)(dd^cu)^q\wedge T.
            \end{equation}	
        \item[(b)] Furthermore, Inequality (\ref{eq 4.1}) holds for  $u,v\in\mathcal{E}^{T}(\Omega)$ such that $u\leq v$ on $\Omega$ and $u=v$ on $\Omega \smallsetminus K$ for some $K\subset \subset \Omega$.
    \end{enumerate}
\end{prop}
\begin{proof}
    $(a)$ Let $u,v\in \mathcal{F}^{T}(\Omega)$ and $u_m,v_j\in\mathcal{E}_{0}^{T}(\Omega)$ which decrease to $u$ and $v$ respectively as in Definition \ref{defn2}. Replace $v_j$ by $\max(u_j,v_j)$ we may assume that $u_j\leq v_j$ for $j\geq1$. By lemma \ref{lem 4} we have for $m\geq j\geq 1$
    $$\ds\frac{1}{q!}\int_{\Omega}(v_j-u_m)^q\wedge(dd^cw)^q \wedge T+\int_{\Omega}(r-w)(dd^cv_j)^q\wedge T \ds\leq\int_{\Omega}(r-w)(dd^cu_m)^q\wedge T.$$
    By  approximating $w$ by a sequence of continuous psh functions vanishing on $\partial \Omega$ (see \cite{Ce2}) and using Proposition \ref{prop 2}, we obtain when $m\rightarrow+\infty$
    $$\ds\frac{1}{q!}\int_{\Omega}(v_j-u)^q\wedge(dd^cw)^q \wedge T+\int_{\Omega}(r-w)(dd^cv_j)^q\wedge T 	 \leq\ds\int_{\Omega}(r-w)(dd^cu)^q\wedge T.$$
    Since $ r-w$ is lower semi-continuous then
    $$\ds\lim_{j\to \infty}\int_{\Omega}(r-w)(dd^cv_j)^q\wedge T\geq \int_{\Omega}(r-w)(dd^cv)^q\wedge T.$$ Hence by tending $j\to+\infty$, we obtain the  result.\\
    $(b)$ Let $G$ and $W$ be open subsets of $\Omega$ such that $K\subset \subset G\subset \subset W \subset \subset \Omega$. There exists $\widetilde{v}\in \mathcal{F}^{T}(\Omega)$ such that $\widetilde{v}\geq v$ on $\Omega$ and $\widetilde{v}=v$ on $W$. Let $\widetilde{u}$ such that $\widetilde{u}=u$ on $G$ and $\widetilde{u}=\widetilde{v}$ either. Since $u=v=\widetilde{v}$ on $W\smallsetminus  K$, we have $\widetilde{u}\in PSH^-(\Omega)$. It follows that $\widetilde{u}\in\mathcal{F}^{T}(\Omega)$, $\widetilde{u}\leq \widetilde{v}$ and $\widetilde{u}=u$ on $W$.\\ Using $(a)$ we obtain
    $$\ds\frac{1}{q!}\int_{\Omega}(\widetilde{v}-\widetilde{u})^q\wedge(dd^cw)^q \wedge T+\int_{\Omega}(r-w)(dd^c\widetilde{v})^q\wedge T\ds\leq\int_{\Omega}(r-w)(dd^c\widetilde{u})^q\wedge T.$$
    As $\widetilde{v}=\widetilde{u}$ on $\Omega\smallsetminus  G$ then
    $$\ds\frac{1}{q!}\int_{W}(\widetilde{v}-\widetilde{u})^q\wedge(dd^cw)^q \wedge T+\int_{W}(r-w)(dd^c\widetilde{v})^q\wedge T\leq\int_{W}(r-w)(dd^c\widetilde{u})^q\wedge T.$$
    Now since $\widetilde{u}=u$, $\widetilde{v}=v$ and $u=v$ on $\Omega\smallsetminus  K$ we obtain
    $$\ds\frac{1}{q!}\int_{\Omega}(v-u)^q \wedge(dd^cw)^q \wedge T+\int_{\Omega}(r-w)(dd^cv)^q\wedge T \ds\leq\int_{\Omega}(r-w)(dd^cu)^q\wedge T.$$
\end{proof}
\begin{rem}
    If we take $w=0$ and $r=1$  in Proposition \ref{prop 5}, we obtain another proof of Proposition \ref{prop 2}.
\end{rem}

\begin{theo}\label{th 7}
    Let $u,w_1,..., w_{q-1}\in \mathcal{F}^{T}(\Omega)$ and $v\in PSH^-(\Omega)$. If we set $S=dd^cw_1\wedge ...\wedge dd^cw_{q-1}$ then
    $$dd^c\max(u,v)\wedge T\wedge S_{|\{u>v\}}=dd^cu\wedge T\wedge S_{|\{u>v\}}.$$
\end{theo}
\begin{proof}
    We prove the theorem in two steps, first  we assume that $v\equiv a < 0$.  Thanks to Lemma \ref{lem 2}, there exist $u_j, w_{k,j}\in\mathcal{E}_{0}^{T}(\Omega)\cap \mathcal{C}(\overline{\Omega})$ such that $(u_j)_j$ decreases to $u$ and  $(w_{k,j})_j$ decreases to $w_k$ for each $1\leq k\leq q-1$.
    Since $\{u_j>a\}$ is open, one has
    $$dd^c\max(u_j,a)\wedge T\wedge S^j_{|\{u_j>a\}}=dd^cu_j\wedge T\wedge S^j_{|\{u_j>a\}}$$
    where $S^j=dd^cw_{1,j}\wedge ...\wedge dd^cw_{q-1,j}$. As $\{u>a\}\subset \{u_j>a\}$ we obtain
    $$dd^c\max(u_j,a)\wedge T\wedge S^j_{|\{u>a\}}=dd^cu_j\wedge T\wedge S^j_{|\{u>a\}}$$
    It follows from \cite{Ha-Du} that
    $$\max(u-a,0)dd^c\max(u_j,a)\wedge T\wedge S^j\underset{j\to+\infty}\longrightarrow \max(u-a,0)dd^c\max(u,a)\wedge T\wedge S$$
    $$\max(u-a,0)dd^cu_j\wedge T\wedge S^j\underset{j\to+\infty}\longrightarrow \max(u-a,0)dd^cu\wedge T\wedge S.$$
    Hence
    $$\max(u-a,0)[dd^c\max(u,a)\wedge T\wedge S-dd^cu\wedge T\wedge S]=0.$$
    So
    $$dd^c\max(u,a)\wedge T\wedge S=dd^cu\wedge T\wedge S\quad on\ \{u>a\}.$$
    Now assume that $v\in PSH^-(\Omega)$. Since $\{u>v\}=\cup_{a\in\Bbb Q^-}\{u>a>v\}$, it suffices to show that
    $$dd^c\max(u,v)\wedge T\wedge S=dd^cu\wedge T\wedge S\quad on\ \{u>a>v\}$$
    for all $a\in\Bbb Q^-$. As $\max(u,v)\in \mathcal{F}^{T}(\Omega)$ then by the first step, we have
    \begin{align*}
        dd^c\max(u,v)\wedge T\wedge S_{|\{\max(u,v)>a\}}&=dd^c\max(\max(u,v),a)\wedge T \wedge S_{|\{\max(u,v)>a\}}\\&=dd^c\max(u,v,a)\wedge T\wedge S_{|\{\max(u,v)>a\}}\\
        dd^cu\wedge T\wedge S_{|\{u>a\}}&=dd^c\max(u,a)\wedge T \wedge S_{|\{v>a\}}.
    \end{align*}
    The fact that  $\max(u,v,a)=\max(u,a)$ on the open set $\{a>v\}$ gives
    $$dd^c\max(u,v,a)\wedge T\wedge S_{|\{a>v\}}=dd^c\max(u,a)\wedge T\wedge S_{|\{a>v\}}.$$
    As  $\{u>a>v$\} is contained in $\{u>a\}$, in $\{\max(u,v)>v\}$ and in $\{a>v\}$, then by combining the last equalities we obtain
    $$dd^c\max(u,v)\wedge T\wedge S_{|\{u>a>v\}}=dd^c\max(u,a)\wedge T\wedge S_{|\{u>a>v\}}.$$
\end{proof}
    We can now prove an inequality analogous to Demailly's one found in \cite{Kh-Ph}.
\begin{prop}\
    \begin{enumerate}
        \item[a)] Let $u,v \in \mathcal{F}^{T}(\Omega)$ such that $(dd^cu)^q\wedge T(\{u=v=-\infty\})=0$ then
            $$(dd^c\max(u,v))^q\wedge T\geq \1_{\{u\geq v\}}(dd^cu)^q\wedge T+\1_{\{u<v\}}(dd^cv)^q\wedge T.$$
        \item[b)] Let $\mu$ be a positive measure vanishing on all pluripolar sets of $\Omega$ and  $u,v \in \mathcal{E}^{T}(\Omega)$ such that $(dd^cu)^q\wedge T\geq \mu$, $(dd^cv)^q\wedge T\geq \mu$. Then $(dd^cmax(u,v))^q\wedge T\geq \mu$.
    \end{enumerate}
\end{prop}
\begin{proof}
    a) For each $\epsilon > 0$ put $A_\epsilon=\{u=v-\epsilon\}\smallsetminus \{u=v=-\infty\}$. Since $A_\epsilon\cap A_\delta= \emptyset$ for $\epsilon\not=\delta$ then there exists $\epsilon_j\searrow0$ such that $(dd^cu)^q\wedge T(A_{\epsilon_j})=0$ for $j\geq 1$. On the other hand, since $(dd^cu)^q\wedge T(\{u=v=-\infty\})=0$ we have $(dd^cu)^q\wedge T(\{u=v-\epsilon_j\})=0$ for $j\geq 1$. Using theorem  \ref{th 7} it follows that
    $$\begin{array}{lcl}
        \ds (dd^c\max(u,v-\epsilon_j))^q\wedge(dd^cw)^q\wedge T\\
        \geq (dd^c\max(u,v-\epsilon_j))^q\wedge T_{|\{u>v-\epsilon_j\}}+(dd^c\max(u,v-\epsilon_j))^q\wedge T_{|\{u<v-\epsilon_j\}}\\
        =(dd^cu)^q\wedge T_{|\{u>v-\epsilon_j\}}+(dd^cv)^q\wedge T_{|\{u<v-\epsilon_j\}}\\
        =\1_{\{u\geq v-\epsilon_j\}}(dd^cu)^q\wedge T+\1_{\{u<v-\epsilon_j\}}(dd^cv)^q\wedge T\\
        \geq\1_{\{u\geq v\}}(dd^cu)^q\wedge T+\1_{\{u<v-\epsilon_j\}}(dd^cv)^q\wedge T.
    \end{array}$$
    Letting $j\to+\infty$ and by Theorem \ref{th 3}, we get
    $$(dd^c\max(u,v))^q\wedge T\geq \1_{\{u\geq v\}}(dd^cu)^q\wedge T+\1_{\{u<v\}}(dd^cu)^q\wedge T$$
    because $\max(u,v-\epsilon_j)\nearrow \max(u,v)$ and $\1_{\{u<v-\epsilon_j\}}\nearrow\1_{\{u<v\}}$ as  $j\to+\infty$.\\
    b) Argument as a).
\end{proof}

\begin{prop}
    Let $u\in\mathcal F^T(\Omega),\ v\in\mathcal E^T(\Omega)$. Then
    $$\begin{array}{l}
        \ds\frac{1}{q!}\int_{\{u<v\}}(v-u)^q\wedge(dd^cw)^q \wedge T+\int_{\{u<v\}}(r-w)(dd^cv)^q\wedge T\\
        \ds \leq\int_{\{u<v\}\cup\{u=v=-\infty\}}(r-w)(dd^cu)^q\wedge T
    \end{array}$$
    for  $w\in PSH(\Omega,[0,1])$ and all $r\geq 1$.
\end{prop}
\begin{proof}
    Let $\varepsilon>0$ and set $\widetilde{v}=\max(u,v-\varepsilon)$. By Inequality (\ref{eq 4.1}) in Proposition \ref{prop 5} we have
    $$\begin{array}{l}
        \ds\frac{1}{q!}\int_{\Omega}(\widetilde{v}-u)^q\wedge(dd^cw)^q \wedge T+\int_{\Omega}(r-w)(dd^c\widetilde{v})^q\wedge T 	 \leq\ds\int_{\Omega}(r-w)(dd^cu)^q\wedge T.
    \end{array}$$
    Since $\{u<\widetilde{v}\}=\{u<v-\varepsilon\}$ then thanks to Theorem \ref{th 7}, we have
    $$\begin{array}{l}
        \ds\frac{1}{q!}\int_{\{u<v-\varepsilon\}}(v-\varepsilon-u)^q\wedge(dd^cw)^q \wedge T+\int_{\{u\leq v-\varepsilon\}}(r-w) (dd^cv)^q\wedge T\\ 	
        \leq\ds\int_{\{u\leq v-\varepsilon\}}(r-w)(dd^cu)^q\wedge T.
    \end{array}$$
    As $\{u\leq v-\varepsilon\}\subset\{u<v\}\cup\{u=v=-\infty\}$ so
    $$\begin{array}{l}
        \ds\frac{1}{q!}\int_{\{u<v-\varepsilon\}}(v-\varepsilon-u)^q\wedge(dd^cw)^q \wedge T+\int_{\{u\leq v-\varepsilon\}}(r-w) (dd^cv)^q\wedge T\\ 	
        \leq\ds\int_{\{u\leq v\}\cup\{u=v=-\infty\}}(r-w)(dd^cu)^q\wedge T.
    \end{array}$$
    Letting $\varepsilon\to0$ we obtain
    $$\begin{array}{l}
        \ds\frac{1}{q!}\int_{\{u<v\}}(v-u)^q\wedge(dd^cw)^q\wedge T+\int_{\{u< v\}}(r-w) (dd^cv)^q\wedge T\\
        \leq\ds\int_{\{u< v\}\cup\{u=v=-\infty\}}(r-w)(dd^cu)^q\wedge T.
    \end{array}$$
\end{proof}

To conclude the proof of the main result, it suffices to take  $w=0$  and $r=1$ in the previous proposition.

\end{document}